\documentclass[12pt,reqno]{amsart}
\usepackage{amssymb,amsmath,amsthm}
\usepackage{color}
\usepackage{pdflscape}
\usepackage{longtable}
\usepackage{enumerate}
\usepackage{tikz}
\usepackage{float}
\usepackage[colorlinks=true,
linkcolor=blue,
urlcolor=blue,
citecolor=blue]{hyperref}
\usepackage{amsaddr}

\usepackage[margin=2cm]{geometry}
\usepackage{tkz-graph}
\usepackage{subcaption}
\usepackage{alltt,xcolor}

\newtheorem{thm}{Theorem}[section]
\newtheorem{lem}[thm]{Lemma}
\newtheorem{prop}[thm]{Proposition}
\newtheorem{hyp}{Hypothesis}
\theoremstyle{definition}
\newtheorem{defn}[thm]{Definition}

\newtheorem{exa}[thm]{Example}

\newtheorem{claim}{Claim}

\newcommand{\itbf}[1]{{\bf{{\emph{{#1}}}}}}
\newcommand{\Aut}[1]{\operatorname{Aut}(#1)}
\newcommand{\End}[2]{\operatorname{End}_{#1}\left(#2\right)}
\newcommand{\sym}[1]{\operatorname{Sym}(#1)}
\newcommand{\alt}[1]{\operatorname{Alt}(#1)}

\newcommand{\agl}[2]{\operatorname{AGL}_#1(#2)}

\newcommand{\triv}[1]{\operatorname{\mathbf{1}}_{#1}}

\begin{document}
	
	\title[]{On the Terwilliger algebras of quasi-thin Schurian association schemes}
	
	\author[R.~Maleki]{Roghayeh Maleki}
	\address{Department of Mathematics and Statistics, University of Regina,\\ 3737 Wascana Parkway, Regina, SK S4S 0A2, Canada}\email{roghayeh.maleki@uregina.ca}
	
	\author[A.~S.~Razafimahatratra]{Andriaherimanana Sarobidy Razafimahatratra$^{*, a}$}\thanks{$^*$ The author gratefully acknowledges that this research was supported by the Fields Institute for Research in Mathematical Sciences.\\ \textsuperscript{a} Corresponding author} 
	\address{School of Mathematics and Statistics, Carleton University,\\ 1125 Colonel by drive, Ottawa, ON K1S 5B6, Canada}\email{sarobidy@phystech.edu}
	
	\date{\today}
	
	
	\begin{abstract} 
		An association scheme is triply transitive if its automorphism group is transitive, and the centralizer algebra of a point stabilizer of its automorphism group coincides with the Terwilliger algebra and its subspace $T_0$. 
		In this paper, we give necessary and sufficient conditions for  a quasi-thin association scheme to be triply transitive. As a by-product, we give new infinite families of triply-transitive association schemes.
	\end{abstract}
	
	\maketitle
	
	\section{Introduction}
	
	Associations schemes are highly-regular combinatorial objects that play a fundamental role in algebraic combinatorics. In 1992, Paul Terwilliger introduced Terwilliger algebras in \cite{terwilliger1992subconstituent,terwilliger1993subconstituent2,terwilliger1993subconstituent3} as a means of studying $P$- and $Q$- polynomial association schemes. The (algebraic) structure of Terwilliger algebras heavily depends on these underlying commutative association schemes, and is generally highly complex. 
	Terwilliger algebras can also be defined, more generally, for non-P or non-Q polynomial commutative association schemes, as well as for non-commutative association schemes as in \cite{BannaiMunemasa95,chen2025structure,muzychuk2016terwilliger,song2017combinatorial}. Association schemes with small valencies are interesting from the viewpoint of Terwilliger algebras since they are often non-commutative, and their Terwilliger algebras are naturally restricted by the valencies. This paper is concerned with quasi-thin association schemes whose Terwilliger algebras are combinatorially and algebraically as small as possible.

	Let $\Omega$ be a finite set and $\mathcal{R} = \{R_0,R_1,\ldots,R_d\}$, with $R_0 = \{ (\omega,\omega): \omega\in \Omega \}$, be a set of relations on $\Omega$. The pair $\mathfrak{X} = (\Omega,\mathcal{R})$ is an \itbf{association scheme} if the following properties hold.
	\begin{enumerate}[({AS}1)]
		\item $\mathcal{R}$ is a partition of $\Omega \times \Omega$.\label{partition}
		\item For $1\leq i\leq d$, the relation $R_{i}^* = \{ (v,u): (u,v) \in R_i \}$ belongs to $\mathcal{R}$. In particular, $*$ can be viewed as a permutation in $\sym{d}$ of order dividing $2$.\label{transposition}
		\item For $0\leq i,j,k\leq d$, there exists a non-negative integer $p_{ij}^k$ such that for $(u,v) \in R_k$ we have\label{intersection-number}
		\begin{align*}
			\left|\left\{ w\in \Omega: (u,w) \in R_i\mbox{ and } (w,v) \in R_j \right\}\right| = p_{ij}^k.
		\end{align*}   
	\end{enumerate}
	In addition, the association scheme $\mathfrak{X}$ is called \itbf{commutative} if $p_{ij}^k = p_{ji}^k$, for all $0\leq i,j,k \leq d$. The digraphs of the scheme $\mathfrak{X}$ are those of the form $(\Omega,R_i)$, where $1\leq i\leq d$.

	If $G\leq \sym{\Omega}$ is a finite transitive group and $\mathcal{O}$ is the set of all orbitals of $G$, then the pair $(\Omega,\mathcal{O})$ satisfies (AS\ref{partition})-(AS\ref{intersection-number}), so it is an association scheme. The association scheme $(\Omega,\mathcal{O})$ is called the \itbf{orbital scheme} of $G\leq \sym{\Omega}$, and its digraphs are called the orbital digraphs. It is well known that the orbital scheme of $G$ is commutative if and only if $(G,G_\omega)$ is a Gelfand pair\footnote{i.e., if $\triv{G_\omega}$ is the trivial character of $G_\omega$, then the induced character $\operatorname{Ind}^G\triv{G_\omega}$ is a sum of distinct irreducible characters of $G$} for any $\omega\in \Omega$. Association schemes that can be obtained as orbitals of certain transitive groups are called \itbf{Schurian}.
	
	For any $1\leq i\leq d$, the pair $(\Omega,R_i)$ is a simple\footnote{No arcs of the form $(u,u)$, for $u\in \Omega$. Moreover, there is no multiple arcs between two vertices.} regular\footnote{i.e., every vertex has equal in-degrees and out-degrees.} digraph if $i^*\neq i$, and is a simple undirected regular graph if $i^* = i$. For any $1\leq i\leq d,$ the \itbf{valency} of the relation $R_i$ is the degree of the graph $(\Omega,R_i)$ if $i^* = i$, and its in-degree if $i^*\neq i$. The valency of the diagonal relation $R_0$ is defined to be $1$. The association scheme $\mathfrak{X}$ is called \itbf{thin} if all valencies are equal to $1$, and is called \itbf{quasi-thin} if the valencies are equal to $1$ or $2$.
	
	Recall that an automorphism of a graph $X$ with vertex set $\Omega$ is a permutation of $\Omega$ preserving the edges and the non-edges of $X$. The set of all automorphisms of $X$ is a permutation group of $\Omega$. The automorphism group $\Aut{\mathfrak{X}}$ of the association scheme $\mathfrak{X} = \left(\Omega,\{R_i\}_{i=0}^d\right)$ is the intersection of all automorphism groups of the digraphs $(\Omega,R_i)$ for $1\leq i\leq d$. For any matrix $A$, let $A(u,v)$ be the $(u,v)$-entry of $A$. The adjacency matrix of $(\Omega,R_i)$, for $1\leq i\leq d$, is the matrix $A_i$ indexed by $\Omega$ in its rows and columns such that
	\begin{align*}
		A_i(u,v) =
		\begin{cases}
			1 & \mbox{ if } (u,v) \in R_i,\\
			0 & \mbox{ otherwise.}
		\end{cases}
	\end{align*}
	We let $A_0$ be the identity matrix indexed by $\Omega$ in its rows and columns.
	
	Let $\mathfrak{X} = (\Omega,\mathcal{R})$ be an association scheme, where $\mathcal{R} = \{R_0,R_1,\ldots,R_d\}$ and $R_0 = \{(u,u):u\in \Omega\}$. Fix $v\in \Omega$ and let $G = \Aut{\mathfrak{X}}$. For any $0\leq i\leq d$, we define the matrix $E_{i,v}^*$ to be the diagonal matrix indexed by $\Omega$ in its rows and columns, where
	\begin{align*}
		E_{i,v}^*(u,u) =
		\begin{cases}
			1 & \mbox{ if } (v,u) \in R_i,\\
			0 & \mbox{ otherwise.}
		\end{cases}
	\end{align*}  
	The Terwilliger algebra of $\mathfrak{X}$ with respect to $v$ is the subalgebra of the complex matrix algebra $\operatorname{M}_{\Omega}(\mathbb{C})$ given by
	\begin{align*}
		T_{v} = \left\langle A_0,A_1,\ldots,A_d, E_{0,v}^*,E_{1,v}^*,\ldots,E_{d,v}^*\right\rangle.
	\end{align*}
	Here, we emphasize on the fact that the Terwilliger algebra $T_v$ depends on the choice of $v$, and two Terwilliger algebras with respect to two different vertices need not be isomorphic.
		
	On the one hand, the Terwilliger algebra $T_v$ is a subalgebra of the centralizer algebra $\End{G_v}{\mathbb{C}^\Omega}$, which is the set of all matrices in $\operatorname{M}_{\Omega}(\mathbb{C})$ commuting with every permutation matrix of $G_v$. On the other hand, the Terwilliger algebra also contains the subspace
	\begin{align}
		T_{0,v} = \operatorname{Span}\left\{ E_{i,v}^*A_jE_{k,v}^*: 0\leq i,j,k\leq d \right\}.
	\end{align}
	Therefore, we have the inclusions
	\begin{align}
		T_{0,v} \subseteq T_v \subseteq  \End{G_v}{\mathbb{C}^\Omega}.\label{eq:inclusion}
	\end{align}
	The gap in dimensions between the subspaces in the above inclusions can be arbitrarily large. One example where such cases occur is with the association scheme of Paley graphs \cite{hanaki2023terwilliger}. Equality can also hold in \eqref{eq:inclusion} as in the case of the association schemes of the Higman-Sims graph and the McLaughlin graph (see \cite{Herman2026}). When $T_{0,v} = T_v =  \End{G_v}{\mathbb{C}^\Omega}$, one may think of the Terwilliger algebra to be as small as possible. Terwilliger algebras with this property admit a highly-regular combinatorial structure given by $T_{0,v} = T_v$, and a highly-symmetric algebraic structure given by $T_v = \End{G_v}{\mathbb{C}^\Omega}$.
	
	The subspace $T_{0,v}$ is an algebra if and only if $T_{0,v} = T_v$. In this case, the association scheme $\mathfrak{X}$ is called \itbf{triply regular.} Examples of triply-regular association schemes are the Hamming scheme $H(n,2)$, the association scheme of the strongly-regular graphs $H(2,d)$, and the affine polar graphs $\operatorname{VO}_{2m}^\varepsilon(2)$, where $\varepsilon = \pm$ and $m\geq 2$. We note that if there exists $v\in \Omega$ such that $T_{0,v}  = T_v$, then $T_{0,u}  = T_u$ for all $u \in \Omega$, whenever the underlying association scheme is $P$-polynomial \cite[Lemma~4.1.2]{guotriply}. 
	
	If $T_v = \End{G_v}{\mathbb{C}^\Omega}$, then the structure of $T_v$ is straightforward since $\End{G_v}{\mathbb{C}^\Omega}$ is completely determined by the action of $G_v$ on $\Omega \times \Omega$. Hence, one can study the Terwilliger algebra from a  purely algebraic viewpoint in this case. 
	
	If the automorphism group $G = \Aut{\mathfrak{X}}$ is transitive, then the Terwilliger algebras $T_u$ and $T_v$ are isomorphic for any $u,v\in \Omega$. Similarly, $\End{G_u}{\mathbb{C}^\Omega}$ and $\End{G_v}{\mathbb{C}^\Omega}$ are isomorphic for any $u,v\in \Omega$ whenever $G$ is transitive. Therefore, for any fixed $v\in \Omega$, we will adopt the notations $T$ for any Terwilliger algebra $T_v$, and $\tilde{T}$ for the centralizer algebra $\End{G_v}{\mathbb{C}^\Omega}$, whenever $G$ is transitive. 
	We note that the dimension of $T_{0,v}$ only depends on the number of non-zero intersection numbers and not on $v\in \Omega$. As we will deal with association schemes with transitive automorphism groups, we fix $v\in \Omega$ and abbreviate $T_{0,v},\ T_v, $ and $ \End{G_v}{\mathbb{C}^\Omega}$ as $T_0,\ T$, and $\tilde{T}$.
	
	We say that an association scheme $\mathfrak{X}$ is \itbf{triply transitive} if $G = \Aut{\mathfrak{X}}$ is transitive and $T_0 = T = \tilde{T}$. Therefore, a triply transitive association scheme is also triply regular. The converse of this statement does not hold in general: the association scheme of many orthogonal array graphs are triply regular but not triply transitive. Triply-transitive association schemes were first studied by Bannai and Munemasa in \cite{BannaiMunemasa95} for conjugacy class association schemes. In particular, they showed that the conjugacy class association scheme of abelian groups are triply transitive. In \cite{maleki2024terwilliger}, the first author extended this result of Bannai and Munemasa to certain metacyclic groups that are ``close'' to being abelian. See \cite{HMR2025,yang2025terwilliger,yang2024terwilliger} for results along the same line of research. In \cite{Herman2026}, triply-transitive association schemes arising from strongly-regular graphs were studied and almost classified except those coming from two families of graphs arising from finite geometry. Not much is known about triply-transitive association schemes except for these few examples (the conjugacy class association schemes and the association schemes of strongly-regular graphs). 
	
	This paper aims to provide more examples of triply-transitive association schemes. A by-product of this work is that we obtain a study of the Terwilliger algebras of quasi-thin association schemes over the complex numbers from a different perspective than those in the literature. We note that there are results on the Terwilliger algebras of such schemes in the literature, the majority of which are over fields of positive characteristic. In \cite{jiang2023terwilliger}, Jiang studied Terwilliger $\mathbb{F}$-algebras of quasi-thin association schemes, where $\mathbb{F}$ is a field of positive characteristic. Recently, Chen and Xi \cite{chen2025structure} studied Terwilliger algebras of quasi-thin association schemes defined over a field of arbitrary characteristic from the viewpoint of representation theory of algebras. In particular, they showed that Terwilliger algebras of quasi-thin association schemes are always quasi-hereditary cellular algebras
	in the sense of Cline-Parshall-Scott and of Graham-Lehrer. 
	The results in our paper are however more from a combinatorial and group-theoretical viewpoints.

	\subsection{Main results}

	If $\mathfrak{X}$ is a thin association scheme, then $\Aut{\mathfrak{X}}$ admits a regular subgroup. It is an easy exercise to show that the subspace $T_0$ in fact coincides with the full matrix algebra $\operatorname{M}_\Omega(\mathbb{C})$. Therefore, any thin association scheme is triply transitive. For the rest of the paper, we will focus on quasi-thin, and non-thin, association schemes. We will consider the following notations for our purpose.
	
	\begin{hyp}
		Consider a quasi-thin association scheme $\mathfrak{X} = (\Omega,\mathcal{R})$ which is not thin, where $\mathcal{R} = \{R_0,R_1,\ldots,R_d\}$, and $R_0 = \{ (v,v): v\in \Omega \}$. For any $1\leq i\leq d$, let $\Gamma_i$ be the digraph $(\Omega,R_i)$. Assume that the automorphism group $G = \Aut{\mathfrak{X}}$ is transitive\footnote{Therefore, Theorem~\ref{thm:Muzychuk-Hirasaka} implies that $\mathfrak{X}$ is Schurian}. In addition,  we adopt the following notations.
		\begin{enumerate}[(a)]
			\item Let $A_0 = I_{|\Omega|}$ be the identity matrix, and let $A_i$ be the adjacency matrix of the digraph $\Gamma_i$, for $1\leq i\leq d$. 
			\item Fix $v\in \Omega$, and denote the orbits of $G_v$ by
			\begin{align*}
				\Delta_i=\Delta_i(v) := \left\{ u\in \Omega: (v,u) \in R_i \right\}
			\end{align*}
			for $0\leq i\leq d$.
			\item Let $T_0, T, \mbox{ and }\tilde{T}$ be respectively the subspace $T_{0,v}$, the algebras $T_v$ and $\End{G_v}{\mathbb{C}^\Omega}$. We let $E_i^* = E_{i,v}^*$, for all $0\leq i\leq d$.
		\end{enumerate}\label{hyp2}
	\end{hyp}
	
	In order to prove the main results, we introduce certain combinatorial structures in quasi-thin association schemes. We call these structures diamond pairs.
	
	\begin{defn}
		Let $\mathfrak{X}$ be an association scheme as in Hypothesis~\ref{hyp2}. We say that a pair $(i,k)$, where $0\leq i,k\leq d$, is a \itbf{diamond pair} of $\mathfrak{X}$ if there exists $0\leq j\leq d$ such that $p_{ij}^k = |\Delta_i| = |\Delta_k|  = 2$. In addition, if $H = G_v$ and $H_{|\Delta_{i}\cup \Delta_{k}}$ is the restriction of $H$ onto $\Delta_{i}\cup \Delta_{k}$, then we say that $(i,k)$ is an \itbf{$H_{|\Delta_{i}\cup \Delta_{k}}$-diamond pair.}\label{def:diamond}
	\end{defn}
	
	As we will see later in Lemma~\ref{lem:stabilizer}, the group $H_{|\Delta_{i}\cup \Delta_{k}}$ is either isomorphic to the cyclic group $\mathbb{Z}_2$ or the Klein group $\mathbb{Z}_2\times \mathbb{Z}_2$. More details on diamond pairs are given later in Section~\ref{sect:diamond}. 
	
	The main result of this paper provides a complete characterization of Terwilliger algebras of quasi-thin association schemes with transitive automorphism group from a combinatorial and algebraic perspective. In particular, we show that triple transitivity is equivalent to the non-existence of $\mathbb{Z}_2$-diamond pairs. 
	
	We state the main result below.
	\begin{thm}
		If $\mathfrak{X}$ be an association scheme as in Hypothesis~\ref{hyp2}, then 
		$\mathfrak{X}$ is triply transitive if and only if it does not admit any $\mathbb{Z}_2$-diamond pairs.
		\label{thm:main1}
	\end{thm}
	
	We provide an infinite family of examples of triply transitive association schemes by using Theorem~\ref{thm:main1}. We note that if $G$ is a group and $h\in G$ is an involution, then the action of $G$ on cosets of $\langle h\rangle$ by left multiplication gives rise to a quasi-thin association scheme.
	
	\begin{thm}
		If $\mathfrak{X}$ is the orbital scheme of a transitive group $G\leq \sym{\Omega}$ with point stabilizer of order $2$, and admits a regular subgroup $R$, then $\mathfrak{X}$ is triply transitive.\label{thm:main}
	\end{thm}
	
	We provide some examples illustrating the application of Theorem~\ref{thm:main} in Examples~\ref{exa1} and \ref{exa2}.
	
	\subsection{Organization of the paper}
	This paper is organized as follows.
	In Section~\ref{sect:background}, we recall some results from permutation group theory and the general theory of Schurian association schemes.  In Section~\ref{sect:diamond}, we give an in-depth analysis of $T_0,T,$ and $\tilde{T}$ whenever the association scheme admits diamond pairs. We prove Theorem~\ref{thm:main1} in Section~\ref{sect:main-result}, and give two examples of how to apply Theorem~\ref{thm:main1} in Section~\ref{sect:examples}. The proof of Theorem~\ref{thm:main} is given in Section~\ref{sect:thm1.2}.
	
	\section{Background results}\label{sect:background}
	\subsection{Permutation group theory}
	Let $\Omega$ be a finite set and $G\leq \sym{\Omega}$ be a transitive group. For any $\omega\in \Omega$, we denote the point stabilizer of $\omega$ in $G$ by $G_\omega$. The permutation group $G\leq \sym{\Omega}$ is permutation equivalent to the group action of $G$ on right cosets of $G_\omega$ by right multiplication. Conversely, a group action is a permutation group whenever the action is faithful\footnote{i.e., the kernel of the action is trivial}.
	
	For any $v \in \Omega$, the \itbf{suborbits} of $G$ with respect to $v$ are the orbits of $G_v$ on $\Omega$. The suborbits of $G$ with respect to $v$ are in one-to-one correspondence with the orbitals of $G$, that is, the orbits of $G$ on $\Omega\times \Omega$. In particular, a suborbit $\Delta$ of $G$ with respect to $v$ corresponds to the neighbours of $v$ in the digraph $(\Omega,O)$, where $O = (v,u)^{G}$, and $u\in \Delta$. The rank of the transitive group $G\leq \sym{\Omega}$ is the number of suborbits of $G$ (or the number of orbitals). 

	\subsection{Algebras and vector spaces related to the Terwilliger algebra}
	Let $\mathfrak{X} = (\Omega,\{R_i\}_{i = 0}^d)$ be an association scheme with automorphism group $G = \Aut{\mathfrak{X}}$. Let $A_i$ be the adjacency matrix of the digraph $(\Omega,R_i)$, for $1\leq i\leq d$. The Bose-Mesner algebra or the adjacency algebra of $\mathfrak{X}$ is the subalgebra $\mathcal{B}$ of $\operatorname{M}_\Omega(\mathbb{C})$ generated by the adjacency matrices $(A_i)_{i=0,1,\ldots,d}$. The Bose-Mesner algebra acts by left multiplication on the vector space $V = \mathbb{C}^\Omega$. Thus, the vector space $V$ is a left $\mathcal{B}$-module. The centralizer algebra of $\mathfrak{X}$ is the algebra $\End{G_v}{V}$ consisting of all matrices in $\operatorname{M}_{\Omega}(\mathbb{C})$ commuting with the elements of $G_v$, viewed as permutation matrices on $\Omega$. 
	
	As explained in the previous section, the Terwilliger algebra $T_v$ is a subalgebra of the centralizer algebra $\End{G_v}{V}$. The next lemma gives the dimension of these centralizer algebras; its proof is omitted as it is a simple exercise.
	
	\begin{lem}
		For any $v\in \Omega$, the dimension of $\End{G_v}{V}$ is equal to the number of orbits of $G_v$ on $\Omega \times \Omega$.
	\end{lem} 
	
	Assume now that the automorphism group $G$ of $\mathfrak{X}$ is transitive. Consequently, we may fix $v\in \Omega$ and use $T_0,T,$ and $ \tilde{T}$ instead of $T_{0,v},T_v,$ and $\End{G_v}{\mathbb{C}^\Omega}$. Let $\Delta_0 = \{v\},\Delta_1,\ldots,\Delta_d$ be the orbits of $G_v$, and for $0\leq i\leq d$, choose a representative $u_i \in \Delta_i$. If $O = (u,w)^{G_v}$ is an orbital of $G_v$, then there exists $0\leq i\leq d$ such that $u \in \Delta_i = u_i^{G_v}$. Therefore, there exists $w^\prime\in \Omega$ such that $O = (u,w)^{G_v} = (u_i,w^\prime)^{G_v}$. We deduce that the number of orbits of $G_v$ on $\Omega \times \Omega$ whose first component is in $\Delta_i$ is equal to the number of orbits of $G_v\cap G_{u_i}$ on $\Omega$. Let $M_i$ be the number of orbits of $G_v\cap G_{u_i}$ on $\Omega$. Then, it is not hard to see that
	\begin{align*}
		\dim (\tilde{T}) = \sum_{i=0}^d M_i.
	\end{align*}
	
	Another important subspace of the Terwilliger algebra $T$ is the subspace $T_0$. Its dimension can be computed using the intersection numbers of $\mathfrak{X}$, which is given in the next lemma.
	\begin{lem}{\cite{BannaiMunemasa95}}
		For any $v\in \Omega$, $E_{i}^*A_jE_{k}^* \neq 0$ if and only if $p_{ij}^k \neq 0$. Moreover, we have
		\begin{align*}
			\dim T_0 = \left|\left\{ (i,j,k): p_{ij}^k \neq 0 \right\}\right|,
		\end{align*}
		and 
		\begin{align*}
			\dim E_{i}^*T_0 E_{k}^* = \left|\left\{ j: p_{ij}^k \neq 0 \right\}\right|,
		\end{align*}
		for $0\leq i,k\leq d$.\label{lem:dim-T0}
	\end{lem}

	\subsection{Triply-transitive association schemes}
	Let $\mathfrak{X} = (\Omega,\{R_i\}_{i = 0}^d)$ be an association scheme with transitive automorphism group $G = \Aut{\mathfrak{X}}$. Fix $v\in \Omega$. For any $0\leq i\leq d$, we define 
	\begin{align*}
		N_i = N_i(v) := \left\{ u\in \Omega: (v,u) \in R_i \right\}.
	\end{align*}
	Let $\Delta_0 = \{v\},\Delta_1,\ldots,\Delta_r$ be the orbits of $G_v$.
	Note that $N_i$ is a union of orbits of $G_v$, and $\mathfrak{X}$ is Schurian if and only if $d = r$, and so up to reordering, we may assume that $N_i = \Delta_i$, for $0\leq i\leq d$.
	As we saw before, $\mathfrak{X}$ is triply transitive if $G$ is transitive and $T_0 = T = \tilde{T}$. We show that such association schemes are Schurian.
	
	\begin{lem}
		If $\mathfrak{X}$ is triply transitive, then it is Schurian.\label{lem:Schurian}
	\end{lem}
	\begin{proof}
		Since $\mathfrak{X}$ is triply transitive, we have 
		\begin{align*}
			\dim E_{i}^* T_0 E_{k}^* = \dim E_{i}^* \tilde{T} E_{k}^*,
		\end{align*}
		for all $0\leq i,k\leq d$. On the one hand, for any $0\leq k\leq d$, we have
		\begin{align*}
			\dim E_{0}^* T_0 E_{k}^* =\left|\left\{ 0\leq j \leq d: p_{0j}^k \neq 0 \right\}\right| = \left|\{k\}\right| = 1.
		\end{align*}
		On the other hand, for any $0\leq k\leq d$, we note that $\dim E_{0}^* \tilde{T} E_{k}^*$ is the number of orbits of $G_v$ on $N_k.$ We deduce that $G_v$ is transitive on every $N_k$, for $0\leq k\leq d$. Hence, $r = d$, and $\mathfrak{X}$ is Schurian.
	\end{proof}
	
	\subsection{Thin and quasi-thin association schemes}
	
	Let $\mathfrak{X} = \left(\Omega,\{R_i\}_{i = 0}^d\right)$ be an association scheme with automorphism group $G = \Aut{\mathfrak{X}}$. If $\mathfrak{X}$ is thin, then it was shown in \cite{evdokimov2000separability} that it is Schurian. Consequently, for any $\omega\in \Omega$, the orbits of $G_\omega$ are of size $1$ and therefore $G$ is a regular group. 
	If $\mathfrak{X}$ is a quasi-thin association scheme however, then it can be non-Schurian. For instance, the quasi-thin association scheme \verb*|as16n173| in \cite{HM} is non-Schurian, and in addition its automorphism group is not transitive. It turns out that this is no coincidence, and holds more generally. Indeed, Hirasaka and Muzychuk \cite{Hirasaka2002,Hirasaka2002a} showed that a quasi-thin association scheme has a transitive automorphism group if and only if it is Schurian. We state this result in the next theorem as we will refer to it later.
	\begin{thm}[Hirasaka-Muzychuk]
		If $\mathfrak{X}$ is a quasi-thin association schemes, then $\Aut{\mathfrak{X}}$ is transitive if and only if $\mathfrak{X}$ is Schurian.\label{thm:Muzychuk-Hirasaka}
	\end{thm}


	For the remainder of this section, we let $\mathfrak{X}$ be a quasi-thin association scheme which is not thin, and admits a transitive automorphism group. By Theorem~\ref{thm:Muzychuk-Hirasaka}, the association scheme $\mathfrak{X}$ is Schurian. Hence, the valencies of $\mathfrak{X}$ are exactly the numbers $|\Delta_0|,|\Delta_1|,\ldots,$ and $|\Delta_d|$.
	Using the fact that the valencies of $\mathfrak{X}$ are $1$ or $2$, we can derive some results on the point stabilizer of $G = \Aut{\mathfrak{X}}$ as follows.
	\begin{lem}
		Let $v\in \Omega$ and denote the subgroup $ G_v$ by $H$. The following statements hold.
		\begin{enumerate}[(i)]
			\item $H$ is an elementary abelian $2$-group.
			\item If $\Delta_i = \{u_1,u_2\}$ and $\Delta_k = \{v_1,v_2\}$ are two orbits of $H$, then the restriction $H_{|{\Delta_i \cup \Delta_k}}$ of $H$ on $\Delta_i \cup \Delta_k$ is isomorphic to $\mathbb{Z}_2$ or $\mathbb{Z}_2\times \mathbb{Z}_2$.
		\end{enumerate}
		\label{lem:stabilizer}
	\end{lem}
	\begin{proof}
		(i) Since $\mathfrak{X}$ is not thin, $H\neq \{1\}$. Given $h\in H$ such that $h\neq 1$, it is clear that $\langle h\rangle$ also has orbits of length $1$ or $2$, and thus $h$ has order equal to $2$. Consequently, every non-identity element of $H$ has order $2$, so $H$ is an elementary abelian $2$-group.
		
		(ii) The restriction $H_{|{\Delta_i \cup \Delta_k}}$ is intransitive on $\Delta_i\cup \Delta_k$, but is transitive on $\Delta_i$ and $\Delta_k$. The only subgroups of $\sym{\Delta_i\cup \Delta_k}$ that have such properties are isomorphic to $\mathbb{Z}_2$ or $\mathbb{Z}_2\times \mathbb{Z}_2$.
	\end{proof}

	\section{Diamond pairs}\label{sect:diamond}
	Diamond pairs were already introduced in Definition~\ref{def:diamond}. In this section, we give an in-depth analysis of diamond pairs and the combinatorial consequences of their existence. Henceforth, we let $\mathfrak{X} = \left(\Omega,\mathcal{R}\right)$ be a quasi-thin association scheme as in Hypothesis~\ref{hyp2}. 
	
	Recall that the pair $(i,k)$, where $0\leq i,k\leq d$, is a {diamond pair} of $\mathfrak{X}$ if there exists $0\leq j\leq d$ such that $p_{ij}^k = |\Delta_i| = |\Delta_k|  = 2$.
	In other words, if $(i,k)$ is a diamond pair, then $\Delta_i = \{u_1,u_2\}$, and $\Delta_k = \{w_1,w_2\}$, and we have the configuration in Figure~\ref{fig:lozenge-triples}. Note that if $(i,k)$ is a diamond pair, then $|\Delta_j| = 2$.
	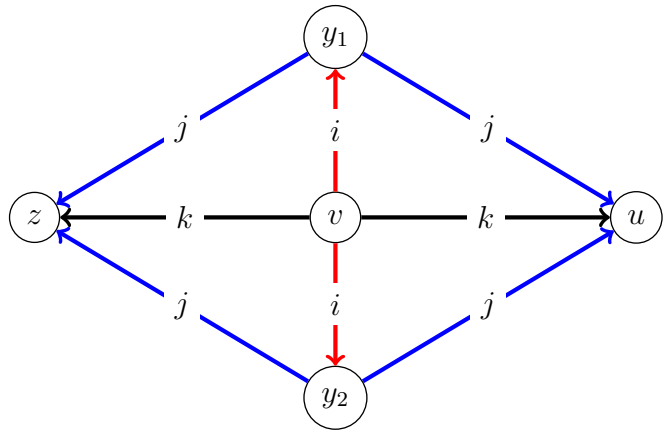
\begin{figure}[H]
		\hspace*{1cm}
		\begin{tikzpicture}[scale=0.4]
			\SetGraphUnit{5}
			\GraphInit[vstyle=Normal]
			\Vertex[x=0,y=0,Math]{v}
			\Vertex[x=10,y=0,Math]{w_1}
			\Vertex[x=0,y=6,Math]{u_1}
			\Vertex[x=0,y=-6,Math]{u_2}
			\Vertex[x=-10,y=0,Math]{w_2}
			\SetUpEdge[style={->,ultra thick},color=black]
			\foreach \v in {w_1,w_2}{%
				\Edge[label=$k$](v)(\v)}
			\SetUpEdge[style={-> ,ultra thick},color=red]
			\foreach \v in {u_1,u_2}{%
				\Edge[label=$i$](v)(\v)}
			\SetUpEdge[style={-> ,ultra thick},color=blue]
			\foreach \v in {u_1,u_2}{%
				\Edge[label=$j$](\v)(w_1)
				\Edge[label=$j$](\v)(w_2)}
		\end{tikzpicture}
		\caption{Configuration for diamond pairs.}\label{fig:lozenge-triples}
	\end{figure}
	If $(i,k)$ is not a diamond pair, then either $|\Delta_i| + |\Delta_k| \leq 3$ or $\Delta_i = \{u_1,u_2\}$ and $\Delta_k = \{w_1,w_2\}$, and $p_{ij}^k\neq 2$ for all $0\leq j\leq d$. By definition of an association scheme, there exists $0\leq j\leq d$ such that $p_{ij}^k \neq 0$. If $p_{ij}^k = 1$, for some $0\leq j\leq d$, then $E_i^*A_jE_k^*$ is non-zero, and using the fact that $\mathfrak{X}$ is quasi-thin, there exist a unique $0\leq j^\prime \leq d$ such that $p_{ij^\prime}^k = 1$ and $E_i^*A_{j^\prime}E_k^*$ is non-zero. Therefore, if $(i,k)$ is not a diamond pair and $|\Delta_{i}| = |\Delta_k| = 2$, then we have the configuration in Figure~\ref{fig:non-diamond} for some $0\leq j,j^\prime \leq d$.
	
	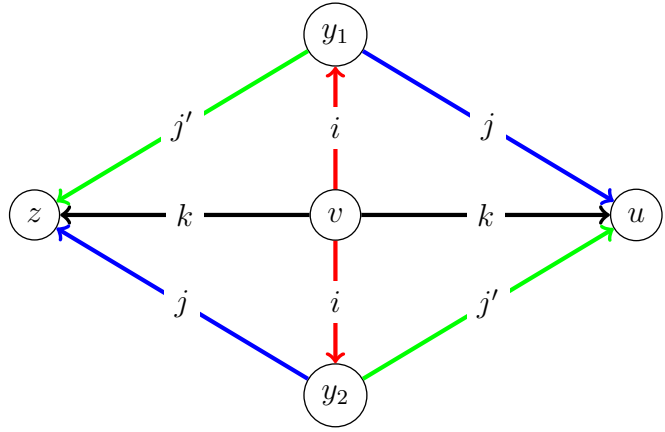
\begin{figure}[H]
			\begin{tikzpicture}[scale=0.4]
				\SetGraphUnit{5}
				\GraphInit[vstyle=Normal]
				\Vertex[x=0,y=0,Math]{v}
				\Vertex[x=10,y=0,Math]{w_1}
				\Vertex[x=0,y=6,Math]{u_1}
				\Vertex[x=0,y=-6,Math]{u_2}
				\Vertex[x=-10,y=0,Math]{w_2}
				\SetUpEdge[style={->,ultra thick},color=black]
				\foreach \v in {w_1,w_2}{
					\Edge[label=$k$](v)(\v)}
				\SetUpEdge[style={-> ,ultra thick},color=red]
				\foreach \v in {u_1,u_2}{%
					\Edge[label=$i$](v)(\v)}
				\SetUpEdge[style={-> ,ultra thick},color=blue]
				\Edge[label=$j$](u_1)(w_1)
				\Edge[label=$j$](u_2)(w_2)
				\SetUpEdge[style={-> ,ultra thick},color=green]
				\Edge[label=$j^\prime$](u_2)(w_1)
				\Edge[label=$j^\prime$](u_1)(w_2)	
			\end{tikzpicture}
		\caption{Configuration for non-diamond pairs.}\label{fig:non-diamond}
	\end{figure}

	By Lemma~\ref{lem:stabilizer}, we know that the restriction of a point stabilizer of a union of two orbits of size $2$ must be isomorphic to the cyclic group $\mathbb{Z}_2$ or the Klein group $\mathbb{Z}_2\times \mathbb{Z}_2$. Therefore, there are two types of diamond pairs, depending on the orbits of point stabilizer $G_v$, which are introduced below. 
	\begin{defn}
		Assume that $(i,k)$ is a diamond pair of $\mathfrak{X}$. Fix $v\in \Omega$ and let $H = G_v$. 
		\begin{itemize}
			\item If $H_{|\Delta_i \cup \Delta_k} \cong \mathbb{Z}_2$, then $(i,k)$ is a \itbf{$\mathbb{Z}_2$-diamond pair}.
			\item If $H_{|\Delta_i \cup \Delta_k} \cong \mathbb{Z}_2 \times \mathbb{Z}_2$, then $(i,k)$ is a \itbf{$\mathbb{Z}_2\times \mathbb{Z}_2$-diamond pair}.
		\end{itemize}
	\end{defn}
	
	Since the properties in the above definition depends on $G_v$, these two types of diamond pairs are indistinguishable from the viewpoint of the subspace $T_0$.
	
	For the remainder of this section, we will exhibit the properties that diamond pairs (in the underlying association scheme) impose on the subspace $T_0$ and the algebra $\tilde{T}$.
	
	\subsection{The subspace $T_0$}
	
	Given a matrix $A\in \operatorname{M}_{\Omega}(\mathbb{C})$, the matrix $E_{i}^*AE_{k}^*$ is a matrix whose entries outside $\Delta_i \times \Delta_k$ are equal to $0$, and those within $\Delta_i \times \Delta_k$ are equal to the entries of $A$. For any $A\in \operatorname{M}_\Omega(\mathbb{C})$, we let $\overline{E_{i}^*A E_{k}^*}$ be the submatrix of $A$ determined by the rows $\Delta_i$ and columns $\Delta_k$. For any $n\geq 2$, we denote the identity matrix and the all-ones matrix of size $n$ by $I_n$ and  $J_n$, respectively.
	
	We state the following propositions whose proofs are straightforward by counting the number of non-zero intersection numbers and using Figure~\ref{fig:lozenge-triples} or Figure~\ref{fig:non-diamond}.
	\begin{prop}
		Let $0\leq i,k\leq d$. If $(i,k)$ is not a diamond pair and $|\Delta_i| = |\Delta_k| = 2$, then 
		there exists a unique pair $(j,j^\prime)$ of distinct integers with $0\leq j,j^\prime \leq d$ such that 
		\begin{align*}
			E_i^* T_0 E_k^* = \operatorname{Span} \{E_i^* A_j E_k^*,E_i^* A_{j^\prime} E_k^*\}.
		\end{align*}
		In particular, $\{\overline{E_i^* A_j E_k^*},\overline{E_i^* A_{j^\prime} E_k^*}\} = \{ I_2, J_2-I_2 \}$, and $\dim E_i^* T_0 E_k^* =2$. \label{prop:non-diamond}
	\end{prop}
	\begin{prop}
		If $(i,k)$ is a diamond pair, then there exists a unique $0\leq j\leq d$ such that $E_i^* T_0 E_k^* = \operatorname{Span} \{E_i^* A_j E_k^*\}$ where $\overline{E_i^*A_jE_k^*} = J_2$. In particular, we have $\dim E_i^* T_0 E_k^* = 1$.\label{prop:diamond}
	\end{prop}
	
	Propositions~\ref{prop:non-diamond} and \ref{prop:diamond} provide an important tool to understanding the dimension of the subspace $T_0$. The case where $|\Delta_{i}|+|\Delta_{k}|\leq 3$ will be considered in the next subsection.

	\subsection{The centralizer algebra}\label{sect:centralizer-algebras}
	

	
	
	We will determine precisely the pair of integers $0\leq i,k\leq d$ such that the subspaces $E_i^*T_0E_k^*$ and $E_i^*\tilde{T} E_k^*$ coincide. For $u,w\in \Omega$, we define the matrix $E_{u,w} \in \operatorname{M}_\Omega(\mathbb{C})$ indexed by $\Omega$ in its rows and columns, such that the $(u,w)$-entry of $E_{u,w}$ is $1$, and every other entry is equal to $0$. For any orbit $O$ of a permutation group $H\leq G$ on $\Omega \times \Omega$, we define the matrix $A(\Omega,O)$ to be the adjacency matrix of the digraph $(\Omega,O)$. If $O_0,O_1,\ldots,O_l$ are the orbits of $G_v$ on $\Omega\times \Omega$, then it is well known that 
	\begin{align*}
		\tilde{T} = \operatorname{Span}\left\{ A(\Omega,O_i): 0\leq i\leq l \right\}.
	\end{align*}
	
	The following lemma will be useful in the proof of the main result of this section.
	\begin{lem}
		For any $0\leq i\leq d$, the orbital $R_i$ of $G$ on $\Omega$ is a union of orbitals of $G_v$ on $\Omega$.\label{lem:orbital-subgroups}
	\end{lem}
	\begin{proof}
		The proof immediately follows from the fact that $G_v$ is a subgroup of $G$.
	\end{proof}
	
	\begin{lem}
		Let $0\leq i,k\leq d$. If $(i,k)$ is not a diamond pair of $\mathfrak{X}$, then 
		\begin{align*}
			E_i^*T_0 E_k^* = E_i^* \tilde{T} E_{k}^*.
		\end{align*}
		Moreover, if $|\Delta_i| = |\Delta_k| = 2$, then $H_{|\Delta_i\cup \Delta_k} \cong \mathbb{Z}_2$, where $H = G_v$.\label{lem:non-diamonds}
	\end{lem}
	\begin{proof}
		Assume that $(i,k)$ is not a diamond pair of $\mathfrak{X}$. To show that $E_i^*T_0 E_k^* = E_i^*\tilde{T} E_k^*$, it is enough to show that $A(\Omega,O) \in T_0$ for all orbitals $O \subseteq \Delta_i\times \Delta_k$ of $G_v$ on $\Omega$.
		\begin{claim}
			If $O \subseteq \Delta_i \times \Delta_k$ is an orbital of $G_v$ on $\Omega$ and $|\Delta_i| = |\Delta_k| = 1$, then $A(\Omega,O)  \in T_0$.\label{claim1}
		\end{claim}
		\begin{proof}[Proof of Claim~\ref{claim1}]
			Clearly, $O = \Delta_{i} \times \Delta_{k}$, and we may assume that $O= \Delta_i\times \Delta_k = \{(u,w)\}$. Note that $A(\Omega,O) = E_{u,w}$, since $|\Delta_i| = |\Delta_k| = 1$. If $0\leq j\leq d$ such that $(u,w)\in R_j$, then 
			\begin{align*}
				E_{i}^* A_j E_{k}^*  = E_{u,w} = A(\Omega,O) \in T_0 .
			\end{align*}
		\end{proof}
		\begin{claim}
			If $O \subseteq \Delta_i \times \Delta_k$ is an orbital of $G_v$ on $\Omega$ and $|\Delta_i| = 1$ and $ |\Delta_k| = 2$, then $A(\Omega,O)  \in T_0$.\label{claim2}
		\end{claim}
		\begin{proof}[Proof of Claim~\ref{claim2}]
			Assume that $\Delta_i = \{u\}$ and $\Delta_k = \{w,w^\prime\}$. Using the fact that $O\subseteq \Delta_i \times \Delta_k$, it is not hard to see that $O = \{(u,w),(u,w^\prime)\}$. Thus, $	A(\Omega,O) = E_{u,w}+E_{u,w^\prime}$. Let $0\leq j\leq d$ such that $(u,w) \in R_j$. Since $O$ is an orbital of $G_v$ on $\Omega$, by Lemma~\ref{lem:orbital-subgroups}, the orbital $R_j$ is a union of orbitals of $G_v$ on $\Omega$, thus including the orbital $O$. Hence, we deduce that $(u,w^\prime) \in R_j$, and consequently
			\begin{align*}
				A(\Omega,O) = E_{u,w}+E_{u,w^\prime} = E_{i}^* A_j E_{k}^* \in  T_0.
			\end{align*}
		\end{proof}
		\begin{claim}
			If $O \subseteq \Delta_i \times \Delta_k$ is an orbital of $G_v$ on $\Omega$ with $|\Delta_i| = 2$ and $ |\Delta_k| = 1$, then $A(\Omega,O)  \in  T_0 $.\label{claim3}
		\end{claim}
		\begin{proof}[Proof of Claim~\ref{claim3}]
			Assume that $\Delta_i = \{u,u^\prime\}$ and $\Delta_k = \{w\}$. Then, it is clear that the orbital $O$ is equal to $\{(u,w),(u^\prime,w)\}$. Hence, $A(\Omega,O) = E_{u,w}+E_{u^\prime,w}$. Let $0\leq j\leq d$ such that $(u,w)\in R_j$. Using Lemma~\ref{lem:orbital-subgroups}, we deduce as in the previous claim that $(u^\prime,w)\in R_j$.
			Then, we have 
			\begin{align*}
				A(\Omega,O) = E_{u,w}+E_{u^\prime,w} = E_{i}^* A_j E_{k}^* \in  T_0.
			\end{align*} 
		\end{proof}
		\begin{claim}
			If $O \subseteq \Delta_i \times \Delta_k$ is an orbital of $G_v$ on $\Omega$ and $|\Delta_i| =  |\Delta_k| = 2$, then $A(\Omega,O)  \in T_0$.\label{claim4}
		\end{claim}
		\begin{proof}[Proof of Claim~\ref{claim4}]
			Let $\Delta_i = \{u_1,u_2 \}$ and $\Delta_k = \{w_1,w_2\}$, and denote $ G_v$ by $H $. By Lemma~\ref{lem:stabilizer}, we know that $H_{|\Delta_i\cup \Delta_k}$ is isomorphic to $\mathbb{Z}_2$ or $\mathbb{Z}_2\times \mathbb{Z}_2$. We will distinguish these two cases.
			
			If $H_{|\Delta_i\cup \Delta_k} \cong \mathbb{Z}_2$, then there are exactly two orbitals of $H$ in $ \Delta_i \times \Delta_k$, which are
			\begin{align*}
				O_1 &= \{(u_1,w_1),(u_2,w_2)\}\mbox{ and } O_2 = \{ (u_1,w_2),(u_2,w_1) \}.
			\end{align*}
			In this case, we note that 
			$
				E_{i}^*\tilde{T} E_{k}^* = \operatorname{Span} \left\{ A(\Omega,O_1),A(\Omega,O_2) \right\}.
			$
			Since $(i,k)$ is not a diamond pair, there exist two distinct integers $0\leq j,j^\prime \leq d$ such that $(u_1,w_1),(u_2,w_2) \in R_j$ and $(u_2,w_1),(u_1,w_2)\in R_{j^\prime}$ (see Figure~\ref{fig:non-diamond}). Hence, we have 
			\begin{align*}
				A(\Omega,O_1) = E_i^* A_jE_k^* \in T_0 \mbox{ and } A(\Omega,O_2) = E_i^* A_{j^\prime}E_{k}^* \in T_0.
			\end{align*}
			Since $O\in \{O_1,O_2\}$, we have $A(\Omega,O)  \in T_0$. 
			
			If $H_{|\Delta_i\cup \Delta_k} \cong \mathbb{Z}_2\times \mathbb{Z}_2$, then $i\neq k$ and there is an involution in $H$ that fixes the two elements of $\Delta_i$ pointwise and swaps the elements of $\Delta_k$. Similarly, there exists another involution in $H$ which fixes the elements of $\Delta_k$ pointwise and swaps the elements of $\Delta_i$. Consequently, the set $\Delta_i \times \Delta_k$ is itself an orbit of $H$ in its action on $\Omega \times \Omega$. Therefore, 
			\begin{align*}
				O = \left\{ (u_1,w_1),(u_2,w_2),(u_1,w_2),(u_2,w_1) \right\}.
			\end{align*}
			By Lemma~\ref{lem:orbital-subgroups}, we note that there exists $0\leq j\leq d$ such that $(u_1,w_1),(u_2,w_2),(u_1,w_2),(u_2,w_1) \in R_j$. Hence, $p_{ij}^k = 2$, and $(i,k)$ must be a diamond pair, which is a contradiction. Therefore,  $H_{|\Delta_i\cup \Delta_k} $ cannot be isomorphic to $ \mathbb{Z}_2\times \mathbb{Z}_2$
		\end{proof}
		With Claims~\ref{claim1}-\ref{claim4}, we have shown that if $(i,k)$ is not a diamond pair, then $	E_i^*T_0 E_k^* = E_i^* \tilde{T} E_{k}^*.$ In the proof of Claim~\ref{claim4}, we also proved that if $|\Delta_i| = |\Delta_k| = 2$, then $H_{|\Delta_i\cup \Delta_k} \cong \mathbb{Z}_2$, where $H = G_v$. This completes the proof.
	\end{proof}
	
	\begin{lem}
		If $(i,k)$ is a $\mathbb{Z}_2\times \mathbb{Z}_2$-diamond pair of $\mathfrak{X}$, then $\dim E_i^* \tilde{T} E_k^* = 1$, and thus $$E_i^* \tilde{T} E_k^* = E_i^* {T}_0 E_k^*.$$ \label{lem:diamond-pairs-Klein}
	\end{lem}
	\begin{proof}
		As $(i,k)$ is a diamond pair, we have seen in Proposition~\ref{prop:diamond} that $\dim E_i^* {T}_0 E_k^* = 1$. So it remains to show that $\dim E_i^* \tilde{T} E_k^* = 1$. Let $H = G_v$ and assume that $\Delta_i = \{u_1,u_2\}$ and $\Delta_k = \{w_1,w_2\}$. Using the fact that $H_{|\Delta_i\cup \Delta_k} \cong \mathbb{Z}_2 \times \mathbb{Z}_2$, we know that there exists an element of $H_{|\Delta_i\cup \Delta_j}$ fixing $\Delta_i$ pointwise, and permuting the elements of $\Delta_k$. Similarly, there exists an element of $H_{|\Delta_i\cup \Delta_k}$ fixing $\Delta_k$ pointwise and permuting $\Delta_i$. Hence, the set $O = \Delta_i \times \Delta_k =\{ (u_1,w_1),(u_1,w_2),(u_2,w_1),(u_2,w_2) \}$ is an orbit of $H$ on $\Omega \times \Omega$. We deduce that 
		\begin{align*}
			E_{i}^*\tilde{T} E_{k}^* = \operatorname{Span} \left\{ A(\Omega,O) \right\}.
		\end{align*}
		Consequently, we have that $\dim E_{i}^*\tilde{T} E_{k}^* = 1$, and $E_{i}^*\tilde{T} E_{k}^* = E_{i}^*{T}_0 E_{k}^*.$
	\end{proof}
	\begin{lem}
		If $(i,k)$ is a $\mathbb{Z}_2$-diamond pair of $\mathfrak{X}$, then $\dim E_i^* \tilde{T} E_k^* = 2$, and $$E_i^* \tilde{T} E_k^* = \operatorname{Span}\{E_i^*AE_k^*,E_i^*BE_k^*\},$$ for some matrices $A,B\in \tilde{T}\setminus T_0$ such that $\overline{E_i^*AE_k^*} = I_2$ and $\overline{E_i^*BE_k^*} = J_2- I_2$.\label{lem:diamond-pairs}
	\end{lem}
	\begin{proof}
		Let $\Delta_i = \{u_1,u_2 \}$ and $\Delta_k = \{w_1,w_2\}$. As $(i,k)$ is a $\mathbb{Z}_2$-diamond pair, the subgroup $H_{|\Delta_i\cup \Delta_k}$ is isomorphic to $\mathbb{Z}_2$. That is, every element of $H = G_v$ that swaps the elements of $\Delta_i$ also swaps the element of $\Delta_k$, and vice versa. Since $|\Delta_i| = |\Delta_k| = 2$, there are two orbitals of $H$ in $ \Delta_i \times \Delta_k$, which are
		\begin{align*}
			O_1 &= \{(u_1,w_1),(u_2,w_2)\}\mbox{ and } O_2 = \{ (u_1,w_2),(u_2,w_1) \}.
		\end{align*}
		Hence, 
		\begin{align*}
			E_{i}^*\tilde{T} E_{k}^* = \operatorname{Span} \left\{ A(\Omega,O_1),A(\Omega,O_2) \right\}.
		\end{align*}
		In particular, we have $\left\{\overline{E_i^*A(\Omega,O_1)E_k^*} ,\overline{E_i^*A(\Omega,O_2)E_k^*}\right\} = \{I_2,J_2-I_2\}$.
		As $(i,k)$ is a diamond pair we know that there exists a unique $0\leq j\leq d$ such that $(u_1,w_1),(u_1,w_2),(u_2,w_1),(u_2,w_2) \in R_j$, and
		\begin{align*}
			E_{i}^* T_0 E_{k}^* = \operatorname{Span} \left\{ E_{i}^*A_jE_{k}^* \right\}
		\end{align*}
		where
		$
			\overline{E_{i}^* A_j E_{k}^*}
			=
			J_2
		$
		. Therefore, $$A(\Omega,O_1),A(\Omega,O_2) \not \in T_0.$$
		This completes the proof.
	\end{proof}
	
%

	
	\section{Proof of Theorem~\ref{thm:main1}}\label{sect:main-result}
	
	In this section, we prove Theorem~\ref{thm:main1}. We first determine when such association schemes are triply regular, then we provide a necessary and sufficient condition for $T$ to be equal to $\tilde{T}$. The proof of Theorem~\ref{thm:main1} follows immediately from these two results.
	
	Throughout this section, we let $\mathfrak{X}$ be an association schemes as in Hypothesis~\ref{hyp2}.
	Recall that $\mathfrak{X}$ is triply regular if $T_0 = T$, or equivalently, $T_0$ is an algebra, so it is closed under multiplication. In the next lemma, we give a necessary and sufficient condition for $\mathfrak{X}$ to be triply regular.
	
	\begin{lem}
		Let $\mathfrak{X}$ be an association scheme as in Hypothesis~\ref{hyp2}. The following statements are equivalent.
		\begin{enumerate}[(i)]
			\item $\mathfrak{X}$ is triply regular.\label{i1}
			\item For any diamond pair $(i,k)$ of $\mathfrak{X}$, there exists no $0\leq \ell\leq d$ with $|\Delta_\ell| = 2$ such that $(i,\ell)$ and $(\ell, k)$ are not diamond pairs.\label{ii2}
		\end{enumerate}\label{lem:triply-regular}
	\end{lem}
	\begin{proof}
		\eqref{i1} $\Rightarrow$ \eqref{ii2} Suppose that $T_0 = T$, or equivalently, $T_0$ is closed under multiplication. Assume by contradiction that there exists a diamond pair $(i,k)$ of $\mathfrak{X}$ and an integer $0\leq \ell \leq d$ with $|\Delta_\ell| = 2$ such that $(i,\ell)$ and $(\ell,k)$ are not diamond pairs. By Proposition~\ref{prop:diamond}, there exists a unique $0\leq j\leq d$ such that $E_i^* T_0 E_k^* = \operatorname{Span} \{E_i^* A_j E_k^*\}$ and $\overline{E_i^*A_jE_k^*} = J_2$. By Proposition~\ref{prop:non-diamond}, we can find $0\leq t,t^\prime \leq d$ such that $\overline{E_i^* A_t E_\ell^*} = I_2$, and $\overline{E_\ell^* A_{t^\prime} E_k^*} = I_2$. Consequently, we have $\left(E_i^* A_t E_\ell^*\right)\left(E_\ell^* A_{t^\prime} E_k^*\right) \not \in E_i^* T_0 E_k^* = \operatorname{Span} \{E_i^* A_j E_k^*\}$. This is impossible since $T_0$ is closed under multiplication. We conclude that \eqref{ii2} holds.
		
		\eqref{i1} $\Leftarrow$ \eqref{ii2} Conversely, assume that for any diamond pair $(i,k)$ of $\mathfrak{X}$, there exists no $0\leq \ell\leq d$ with $|\Delta_\ell| = 2$ such that $(i,\ell)$ and $(\ell, k)$ are not diamond pairs. Assume by contradiction that $T_0$ is not closed under multiplication. Then, there exist three integers $0\leq j,j^\prime,\ell \leq d$ such that $\left(E_{i}^*A_jE_\ell^*\right) \left(E_{\ell}^*A_{j^\prime}E_{k}^*\right) \not \in T_0$.

		If $(i,k)$ is not a diamond pair, then by Lemma~\ref{lem:non-diamonds}, we know that $E_i^*T_0 E_k^* = E_i^* \tilde{T} E_{k}^*.$ Using the fact that $\left(E_{i}^*A_jE_\ell^*\right) \left(E_{\ell}^*A_{j^\prime}E_{k}^*\right) \in E_i^*\tilde{T}E_k^*$, we have $\left(E_{i}^*A_jE_\ell^*\right) \left(E_{\ell}^*A_{j^\prime}E_{k}^*\right) \in E_i^*\tilde{T}E_k^* =E_i^*T_0 E_k^* \subseteq T_0$, which is impossible. Therefore, $(i,k)$ is a diamond pair and $|\Delta_{i}| = |\Delta_{k}| = 2$. If $|\Delta_\ell| = 1$, then $(i,\ell)$ and $(\ell, k)$ are not diamond pairs, and we have 
		\begin{align*}
			\overline{(E_i^*A_jE_\ell^*)(E_\ell^* A_{j^\prime} E_k^*)}
			=
			\begin{bmatrix}
				1 \\
				1
			\end{bmatrix}
			\begin{bmatrix}
				1 &1
			\end{bmatrix}
			=
			\begin{bmatrix}
				1 & 1\\
				1 & 1
			\end{bmatrix},
		\end{align*}
		implying that $(E_i^*A_jE_\ell^*)(E_\ell^* A_{j^\prime} E_k^*) \in T_0$, a contradiction. Therefore, $|\Delta_\ell| = 2$.		
		By \eqref{ii2}, we know that for $0\leq \ell \leq d$ with $|\Delta_\ell| = 2$, one of $(i,\ell)$ or $(\ell,k)$ is a diamond pair. Without loss of generality, assume that $(i,\ell)$ is a diamond pair. By Proposition~\ref{prop:diamond}, there exists a unique integer $0\leq j\leq d$ such that $E_{i}^*A_jE_\ell^* \neq 0 $, and $\overline{E_{i}^*A_jE_\ell^*} = J_2$. As 
		\begin{align*}
			\overline{E_\ell^*A_{j^\prime}E_k^*} \in 
			\left\{
			J_2,
			I_2,
			J_2- I_2
			\right\},
		\end{align*}
		we have  $\overline{E_{i}^*A_jE_\ell^*}\overline{E_\ell^*A_{j^\prime}E_k^*}$ is a multiple of $ J_2$. We conclude that $\left(E_{i}^*A_jE_\ell^*\right) \left(E_{\ell}^*A_{j^\prime}E_{k}^*\right) \in E_i^* T_0E_k^* = \operatorname{Span} \{ E_i^* A_jE_k^* \} \subseteq T_0$, which is a contradiction. Hence, $T_0$ must be closed under multiplication, and so $T_0 = T$.
	\end{proof}
	
	\begin{lem}
		Let $\mathfrak{X}$ be an association scheme as in Hypothesis~\ref{hyp2}. The following statements are equivalent.
		\begin{enumerate}[(a)]
			\item $ T = \tilde{T}$.\label{a}
			\item For any  $\mathbb{Z}_2$-diamond pair $(i,k)$ of $\mathfrak{X}$, there exists an integer $0\leq \ell\leq d$ with $|\Delta_\ell| = 2$ such that $(i,\ell)$ and $(\ell,k)$ are not diamond pairs.\label{b}
		\end{enumerate}\label{lem:TandTtilde}
	\end{lem}
	\begin{proof}

		\eqref{a} $\Rightarrow$ \eqref{b}. Assume that $T = \tilde{T}$.
		If $(i,k)$ is a $\mathbb{Z}_2$-diamond pair, then we know that $\dim E_i^* T_0E_k^* = 1$ and by \eqref{a}, $\dim E_i^*TE_k^* = \dim E_i^* \tilde{T} E_k^* = 2$. 
		We deduce that there exist $0\leq \ell,j,j^\prime\leq d$ such that $(E_i^*A_jE_\ell^*)(E_\ell^* A_{j^\prime} E_k^*) \not\in T_0$. If $|\Delta_\ell| = 1$, then $(i,\ell)$ and $(\ell, k)$ are not diamond pairs, and we have 
		\begin{align*}
			\overline{(E_i^*A_jE_\ell^*)(E_\ell^* A_{j^\prime} E_k^*)}
			=
			\begin{bmatrix}
				1 \\
				1
			\end{bmatrix}
			\begin{bmatrix}
				1 &1
			\end{bmatrix}
			=
			\begin{bmatrix}
				1 & 1\\
				1 & 1
			\end{bmatrix},
		\end{align*}
		implying that $(E_i^*A_jE_\ell^*)(E_\ell^* A_{j^\prime} E_k^*) \in T_0$, a contradiction.
		Hence, $|\Delta_\ell| = 2$. If one of $(i,\ell)$ or $(\ell,k)$ is a diamond pair, then we either have $\overline{E_i^*A_jE_\ell^*} = J_2$ or $\overline{E_\ell^* A_{j^\prime} E_k^*} = J_2$. It is not hard to see that in either case, $(E_i^*A_jE_\ell^*)(E_\ell^* A_{j^\prime} E_k^*) \in E_i^*T_0E^*_k$, which is a contradiction again. Therefore, none of $(i,\ell)$ and $(\ell,k)$ are diamond pairs.
		
		\eqref{b} $\Rightarrow$ \eqref{a} It is enough to show that $E_i^* TE_k^* = E_i^* \tilde{T} E_k^*$ for all $0\leq i,k\leq d$. If $(i,k)$ is not a diamond pair, then by Proposition~\ref{prop:non-diamond} and Lemma~\ref{lem:non-diamonds}, we have $E_i^*T_0E_k^* = E_{i}^*\tilde{T} E_k^*$, and in particular, $E_i^*TE_k^* = E_{i}^*\tilde{T} E_k^*$.
		
		Assume that $(i,k)$ is a diamond pair. Then, $\dim E_i^*T_0E_k^* = 1$. If $(i,k)$ is a $\mathbb{Z}_2\times \mathbb{Z}_2$-diamond pair, then by Lemma~\ref{lem:diamond-pairs-Klein}, we know that $\dim E_i^* \tilde{T} E_k^* = 1$, so it is clear that $E_i^* T_0E_k^* = E_i^* TE_k^* = E_i^* \tilde{T} E_k^*$. If $(i,k)$ is a $\mathbb{Z}_2$-diamond pair, then by \eqref{b} there exists an $0\leq \ell\leq d$ such that $(i,\ell)$ and $(\ell, k)$ are not diamond pairs. So, there exist $0\leq j,j^\prime \leq d$ such that $(E_i^*A_jE_\ell)(E_\ell A_{j^\prime} E_k^*) \in T\setminus T_0$. In particular, $1 = \dim E_i^*T_0 E_k^* < \dim E_i^* T E_k^* \leq  \dim E_i^* \tilde{T} E_k^* = 2$, so we must have $E_i^* TE_k^* = E_i^*\tilde{T}E_k^*$, for every $0\leq i,k\leq d$ such that $(i,k)$ is a $\mathbb{Z}_2$-diamond pair. Consequently, $T = \tilde{T}$.
		
	\end{proof}
	
	Now, we are ready to prove Theorem~\ref{thm:main1}.
	\begin{proof}[{\bf Proof of Theorem~\ref{thm:main1}}]
		
		If $\mathfrak{X}$ is triply transitive, then $T_0 = T = \tilde{T}$. By Lemma~\ref{lem:triply-regular} and Lemma~\ref{lem:TandTtilde}, $\mathfrak{X}$ cannot admit $\mathbb{Z}_2$-diamond pairs. Conversely, if $\mathfrak{X}$ does not admit $\mathbb{Z}_2$-diamond pair, then Lemma~\ref{lem:TandTtilde}\eqref{b} vacuously hold, so we have $ T = \tilde{T}$. By assumption, $\mathfrak{X}$ admits only $\mathbb{Z}_2\times \mathbb{Z}_2$-diamond pairs, so, it is enough to show that Lemma~\ref{lem:triply-regular}\eqref{ii2} holds for $\mathbb{Z}_2\times \mathbb{Z}_2$-diamond pairs. By contradiction, assume that \eqref{ii2} does not hold. Then, there exists a $\mathbb{Z}_2\times \mathbb{Z}_2$-diamond pair $(i,k)$ of $\mathfrak{X}$ and an integer $0\leq \ell\leq d$ with $|\Delta_\ell| = 2$ such that $(i,\ell)$ and $(\ell, k)$ are not diamond pairs. This implies $E_i^* T_0 E_{k}^*$ is not closed under multiplication, which is a contradiction to Lemma~\ref{lem:diamond-pairs-Klein}, so \eqref{ii2} holds and $T_0 = T$.
	\end{proof}
	
	\section{Examples of triply-transitive quasi-thin association schemes}\label{sect:examples}
	Theorem~\ref{thm:main1} gives a characterization of the triply-transitive quasi-thin association schemes by counting the number of diamond pairs. We give two examples of quasi-thin association schemes admitting transitive automorphism groups that are not triply transitive. These two examples are in fact only such examples with order up to $30$.

	
	We will see examples arising from transitive permutation group actions of a group by right multiplication on the subgroup generated by an involution. 
	If $G$ is a group and $H = \langle b\rangle \leq G$ has order $2$, then the action of $G$ on right cosets of $H$ by right multiplication is faithful, and is therefore permutation equivalent to a transitive permutation group of degree $[G:H]$. By choice of $H$, the orbital scheme arising from this transitive group is quasi-thin, and if $\Delta_i$ and $\Delta_k$ are two orbits of $H$ of size $2$, then $H_{|\Delta_i \cup \Delta_{k}} = \mathbb{Z}_2$.


	\begin{exa}
		Consider the group $$G = \alt{5} = \left\langle a,b| a^5=b^2 = (ab)^3 =  1 \right\rangle$$ and its subgroup $H = \langle b\rangle$. There is a unique conjugacy class of subgroups of order $2$ in $G$, so all actions of $G$ on cosets of a subgroup of order $2$ are permutation equivalent to $G$ acting on cosets of $H$. This group is permutation equivalent to \verb*|TransitiveGroup(30,9)| in the library of transitive groups of \verb*|Sagemath| \cite{sagemath}.
		
		If $g\in G$ has a fixed point in this action, then $g$ is conjugate in $H$. Therefore, the only permutations in $G$ fixing a coset of $H$ must be the identity or an involution. Moreover, the number of fixed cosets is either equal to $[G:H] = 30$ or the number of orbits of size $1$ of $H$(equivalently, the number of cosets fixed by $b$). If $(Hx)b = Hx$ for some $x\in G$, then we have $xbx^{-1} \in H$. In other words, if $Hx$ is fixed by $b$, then $x\in \operatorname{N}_{G}(\langle b\rangle)$. It is not hard to see that
		\begin{align*}
			\operatorname{N}_G(\langle b\rangle) = \operatorname{C}_{G}(\langle b\rangle) =  \langle b \rangle \times \mathbb{Z}_2.
		\end{align*} 
		Therefore, $b$ fixes two cosets of $H$. A direct consequence of this fact is that the rank of $G$ acting by right multiplication on right cosets of $H$ is $1+1+\tfrac{28}{2} = 16$, and
		\begin{align*}
			\dim \tilde{T} = 2\times 16+ 14\times 30 = 452.
		\end{align*}
		
		As the rank of $G$ acting on cosets of $H$ is $1+1+14=16$, we may assume that the orbits of $H$ are $\Delta_1= \{H\}$, $\Delta_2 = \{ Hx_2 \}$, $\Delta_3 = \{Hx_{3},Hy_{3}\}$, $\ldots$, and $\Delta_{16} = \{Hx_{16},Hy_{16}\}$. 
		
		Since the scheme is quasi-thin, we know that $p_{ij}^k\leq 2$, for all $0\leq i,j,k\leq d$. There exists a unique pair $(i,k)$ where $0\leq i,k\leq d$ and $|\Delta_{i}| = |\Delta_{k}| = 2$ such that $p_{ij}^k = p_{kj}^i= 2$. Hence, these are $\mathbb{Z}_2$-diamond pairs.
		The association scheme can be represented by the matrix given below. The entries in {\color{blue} blue} and {\color{red} red} correspond to the pairs of orbits forming $\mathbb{Z}_2$-diamond pairs.
		{\footnotesize
			\begin{alltt}
				[ 0  1  2  3  4  5  5  4  6  7  3  2  8  9 10 11 12 13  8  9 14 14 15 15 11 10 13 12  6  7]
				[ 1  0 13 12  9  8  8  9 10 11 12 13  5  4  6  7  3  2  5  4 14 14 15 15  7  6  2  3 10 11]
				[ 2  3  0  1  4  5  2  3 14 14  4  5 11 10  9  8  6  7 15 15  7  6  8  9 10 11 12 13 12 13]
				[13 12  1  0  9  8 13 12 14 14  9  8  7  6  4  5 10 11 15 15 11 10  5  4  6  7  3  2  3  2]
				[ 8  9  8  9  0  1 13 12  2  3 12 13  4  5  5  4 11 10  6  7  3  2  6  7 15 15 14 14 11 10]
				[ 5  4  5  4  1  0  2  3 13 12  3  2  9  8  8  9  7  6 10 11 12 13 10 11 15 15 14 14  7  6]
				[ 5  4  2  3  3  2  0  1  7  6  4  5 10 11 15 15 14 14  9  8 13 12 11 10  9  8  6  7 13 12]
				[ 8  9 13 12 12 13  1  0 11 10  9  8  6  7 15 15 14 14  4  5  2  3  7  6  4  5 10 11  2  3]
				[11 10 15 15  2  3  7  6  0  1 13 12  4  5  2  3  8  9 14 14  4  5 12 13  8  9  6  7 10 11]
				[ 7  6 15 15 13 12 11 10  1  0  2  3  9  8 13 12  5  4 14 14  9  8  3  2  5  4 10 11  6  7]
				[13 12  8  9 12 13  8  9  3  2  0  1 15 15  7  6  2  3  7  6 10 11  4  5  5  4 11 10 14 14]
				[ 2  3  5  4  3  2  5  4 12 13  1  0 15 15 11 10 13 12 11 10  6  7  9  8  8  9  7  6 14 14]
				[ 4  5  6  7  8  9 10 11  8  9 14 14  0  1 13 12  4  5  2  3 12 13 11 10  6  7  2  3 15 15]
				[ 9  8 10 11  5  4  6  7  5  4 14 14  1  0  2  3  9  8 13 12  3  2  7  6 10 11 13 12 15 15]
				[10 11  9  8  5  4 14 14  2  3  7  6  3  2  0  1 15 15  7  6  4  5 13 12 11 10 12 13  9  8]
				[ 6  7  4  5  8  9 14 14 13 12 11 10 12 13  1  0 15 15 11 10  9  8  2  3  7  6  3  2  4  5]
				[12 13 11 10  6  7 15 15  4  5  2  3  8  9 14 14  0  1 13 12 11 10  4  5  2  3  8  9  7  6]
				[ 3  2  7  6 10 11 15 15  9  8 13 12  5  4 14 14  1  0  2  3  7  6  9  8 13 12  5  4 11 10]
				[ 4  5 14 14 11 10  9  8 15 15  7  6  2  3  7  6  3  2  0  1 13 12 10 11 12 13  5  4  8  9]
				[ 9  8 14 14  7  6  4  5 15 15 11 10 13 12 11 10 12 13  1  0  2  3  6  7  3  2  8  9  5  4]
				[15 15  7  6 13 12  3  2  8  9 10 11 12 13  8  9  6  7  3  2  0  1 \textcolor{red}{14} \textcolor{red}{14}  4  5 11 10  5  4]
				[15 15 11 10  2  3 12 13  5  4  6  7  3  2  5  4 10 11 12 13  1  0 \textcolor{red}{14} \textcolor{red}{14}  9  8  7  6  8  9]
				[14 14  4  5 11 10  6  7 12 13  8  9  6  7  3  2  8  9 10 11 \textcolor{blue}{15} \textcolor{blue}{15}  0  1 13 12  4  5  3  2]
				[14 14  9  8  7  6 10 11  3  2  5  4 10 11 12 13  5  4  6  7 \textcolor{blue}{15} \textcolor{blue}{15}  1  0  2  3  9  8 12 13]
				[ 6  7 10 11 14 14  9  8  4  5  5  4 11 10  6  7  2  3 12 13  8  9  3  2  0  1 15 15 13 12]
				[10 11  6  7 14 14  4  5  9  8  8  9  7  6 10 11 13 12  3  2  5  4 12 13  1  0 15 15  2  3]
				[ 3  2 12 13 15 15 11 10 11 10  6  7  2  3 12 13  4  5  5  4  6  7  8  9 14 14  0  1  9  8]
				[12 13  3  2 15 15  7  6  7  6 10 11 13 12  3  2  9  8  8  9 10 11  5  4 14 14  1  0  4  5]
				[11 10 12 13  6  7  3  2 10 11 15 15 14 14  9  8  7  6  4  5  5  4 13 12  3  2  9  8  0  1]
				[ 7  6  3  2 10 11 12 13  6  7 15 15 14 14  4  5 11 10  9  8  8  9  2  3 12 13  4  5  1  0]
			\end{alltt}
		}
		
		The subspace $T_0$ coincides with $\tilde{T}$ everywhere except on the orbits corresponding to the blocks in blue and red. In  each of these blocks, the difference in dimension  between the block of $T_0$ and $\tilde{T}$ is $1$. We deduce that $\dim T_0 = 450$, so $\mathfrak{X}$ is not triply transitive. Using Lemma~\ref{lem:TandTtilde}, we also deduce that $T = \tilde{T}$.
	\end{exa}
	\begin{exa}
		Let $G = \agl{1}{8}$. Then, $G$ acts doubly transitively on the field $\mathbb{F}_8$ as affine transformations. Therefore, $G$ also acts transitively on the $2$-subsets of $\mathbb{F}_8$. This transitive permutation group is \verb*|TransitiveGroup(28,11)| in the library of transitive groups in \verb*|Sagemath|. The stabilizer of a $2$-subsets of $\mathbb{F}_8$ is therefore a cyclic group of order $2$, and the orbital scheme is quasi-thin. The resulting orbital scheme is given below, where the entries highlighted in colours correspond to $\mathbb{Z}_2$-diamond pairs.
		
		\newpage
		{\footnotesize
		\begin{alltt}
			[ 0  1  2  2  3  4  5  5  6  7  8  9 10 11 12 13 14 14  3  4 15 15  7  6  9  8 10 11]
			[ 1  0 14 14  4  3 15 15  6  7  9  8 10 11 13 12  2  2  4  3  5  5  7  6  8  9 10 11]
			[10 11  0  1  2  2  4  3 15 15  7  6  8  9 10 11 12 13 14 14  4  3  5  5  6  7  9  8]
			[10 11  1  0 14 14  3  4  5  5  7  6  9  8 10 11 13 12  2  2  3  4 15 15  6  7  8  9]
			[ 8  9 10 11  0  1  \textcolor{red}{2}  \textcolor{red}{2}  3  4  5  5  7  6  9  8 10 11 12 13 14 14  3  4 15 15  6  7]
			[ 9  8 10 11  1  0 14 14  4  3 15 15  7  6  8  9 10 11 13 12  \textcolor{red}{2}  \textcolor{red}{2}  4  3  5  5  6  7]
			[ 7  6  9  8 10 11  0  1  2  2  4  3 15 15  6  7  8  9 10 11 12 13 14 14  4  3  5  5]
			[ 7  6  8  9 10 11  1  0 14 14  3  4  5  5  6  7  9  8 10 11 13 12  2  2  3  4 15 15]
			[15 15  6  7  8  9 10 11  0  1  2  2  \textcolor{green}{4}  \textcolor{blue}{3}  5  5  7  6  9  8 10 11 12 13 14 14  \textcolor{green}{4}  \textcolor{blue}{3}]
			[ 5  5  6  7  9  8 10 11  1  0 14 14  \textcolor{blue}{3}  4 15 15  7  6  8  9 10 11 13 12  2  2  \textcolor{blue}{3}  4]
			[ 3  4  \textcolor{orange}{5}  \textcolor{orange}{5}  7  6  9  8 10 11  0  1  2  2  3  4 15 15  6  7  8  9 10 11 12 13 14 14]
			[ 4  3 15 15  7  6  8  9 10 11  1  0 14 14  4  3  5  5  6  7  9  8 10 11 13 12  2  2]
			[ 2  2  3  4  5  5  6  7  9  8 10 11  0  1 14 14  3  4 15 15  7  6  8  9 10 11 12 13]
			[14 14  4  3 15 15  6  7  8  9 10 11  1  0  2  2  4  3  5  5  7  6  9  8 10 11 13 12]
			[12 13  2  2  4  3 15 15  7  6  8  9 11 10  0  1 14 14  4  3  5  5  6  7  9  8 11 10]
			[13 12 14 14  3  4  5  5  7  6  9  8 11 10  1  0  2  2  3  4 15 15  6  7  8  9 11 10]
			[11 10 12 13  2  2  3  4  5  5  6  7  8  9 11 10  0  1 14 14  3  4 15 15  7  6  9  8]
			[11 10 13 12 14 14  4  3 15 15  6  7  9  8 11 10  1  0  2  2  4  3  5  5  7  6  8  9]
			[ 8  9 11 10 12 13  \textcolor{red}{2}  \textcolor{red}{2}  4  3 15 15  6  7  9  8 11 10  0  1 14 14  4  3  5  5  7  6]
			[ 9  8 11 10 13 12 14 14  3  4  5  5  6  7  8  9 11 10  1  0  \textcolor{red}{2}  \textcolor{red}{2}  3  4 15 15  7  6]
			[ 6  7  9  8 11 10 12 13  2  2  3  4  5  5  7  6  8  9 11 10  0  1 14 14  3  4 15 15]
			[ 6  7  8  9 11 10 13 12 14 14  4  3 15 15  7  6  9  8 11 10  1  0  2  2  4  3  5  5]
			[ 5  5  7  6  8  9 11 10 12 13  2  2  \textcolor{blue}{3}  4 15 15  6  7  9  8 11 10  0  1 14 14  \textcolor{blue}{3}  4]
			[15 15  7  6  9  8 11 10 13 12 14 14  \textcolor{green}{4}  \textcolor{blue}{3}  5  5  6  7  8  9 11 10  1  0  2  2  \textcolor{green}{4}  \textcolor{blue}{3}]
			[ 4  3 15 15  6  7  9  8 11 10 12 13  2  2  4  3  5  5  7  6  8  9 11 10  0  1 14 14]
			[ 3  4  \textcolor{orange}{5}  \textcolor{orange}{5}  6  7  8  9 11 10 13 12 14 14  3  4 15 15  7  6  9  8 11 10  1  0  2  2]
			[ 2  2  4  3 15 15  7  6  9  8 11 10 12 13 14 14  4  3  5  5  6  7  8  9 11 10  0  1]
			[14 14  3  4  5  5  7  6  8  9 11 10 13 12  2  2  3  4 15 15  6  7  9  8 11 10  1  0]
		\end{alltt}
		}
		
		There are exactly 376 non-zero intersection numbers, so $\dim T_0 = 376$. In total, there are 24 $\mathbb{Z}_2$-diamond pairs in the scheme, so $\dim \tilde{T} = 376+24 = 400$. Moreover, the scheme also satisfies the conditions of Lemma~\ref{lem:TandTtilde}, so we  have $T = \tilde{T}$.
	\end{exa}
	\section{Proof of Theorem~\ref{thm:main}}\label{sect:thm1.2}
	
	For convenience, we consider the following hypothesis.
	
	\begin{hyp}
		Let $\mathfrak{X}$ be an association scheme as in Hypothesis~\ref{hyp2}.
		We assume additionally that $G$ has the following properties.
		\begin{enumerate}[(a)]
			\item The group $G$ admits a subgroup $R$ acting regularly.\label{hyp-Cayley}
			\item The stabilizer $G_v $ is a subgroup generated by an involution $ \varphi$.
		\end{enumerate}\label{hyp}
	\end{hyp}
	
	Theorem~\ref{thm:main} is restated as follows for convenience.
	\begin{thm}
		If $\mathfrak{X}$ is an association scheme as in Hypothesis~\ref{hyp}, then it is triply transitive.
		
	\end{thm}
	
	Let $\mathfrak{X} $ be an association scheme as defined in Hypothesis~\ref{hyp}. We will first derive a few important results form Hypothesis~\ref{hyp}. For any $1\leq i\leq d$, we let $\Gamma_i$ be the digraph $(\Omega,R_i)$.	
	By Hypothesis~\ref{hyp}\eqref{hyp-Cayley}, we know that $R\leq \Aut{\Gamma_i}$ acts regularly for all $1\leq i\leq d$. Hence, $\Gamma_i$ is a Cayley digraph for all $1\leq i\leq d$, and we may identify $\Omega$ and $R$. In particular, there is a one-to-one correspondence identifying $v$ to the identity element $1\in R$, which is a graph isomorphism.

	We will assume from hereon that $\Delta_0 = \{1\}$ and $\Delta_i \subset R$ for all $1\leq i\leq d$.
	Consequently, we may assume without loss of generality that for $1\leq i\leq d$
	\begin{align*}
		\Gamma_i = \operatorname{Cay}(R,\Delta_i).
	\end{align*}
	Since $\mathfrak{X}$ is Schurian, $\Gamma_i$ is an orbital digraph of $G$ for any $1\leq i\leq d$. The orbits of the stabilizer $G_1$ of $1\in R$ in $G$ are therefore the sets $\Delta_0 = \{1\},\Delta_1,\ldots,\Delta_d$. Assume that $G_1 = \langle \varphi\rangle$, where $\varphi$ is the involution in Hypothesis~\ref{hyp}. 
	
	We claim that the group $G $ is a split extension of $ R$. Indeed, since $R\leq G$ is a regular subgroup, we have $|R| = |\Omega| = \tfrac{|G|}{|G_1|} = \tfrac{|G|}{2}$ and  $[G:R] = 2$.
	As $R\trianglelefteq G$ is a maximal subgroup, and $\varphi \not\in R$, it is clear that $G = R\langle \varphi \rangle = RG_1	$. 
	Consequently, $G$ is a split extension of $R$, and $G = R\rtimes_\Psi G_1$ for some homomorphism $\Psi: G_1 \to \Aut{R}$. 
	
	\begin{proof}[{\bf Proof of Theorem~\ref{thm:main}}]
		By Theorem~\ref{thm:main1}, it is enough to show that $\mathfrak{X}$ does not admit $\mathbb{Z}_2$-diamond pairs. If $(i,k)$ is a $\mathbb{Z}_2$-diamond pair of $\mathfrak{X}$, and $\Delta_i = \{u_{1},u_{2}\}$ and $\Delta_k = \{w_1,w_2\}$, then we have $(u_1,w_1),(u_2,w_1) \in R_j$ as illustrated below. 
		
		\begin{figure}[H]
			\hspace*{1cm}
			\begin{tikzpicture}[scale=0.4]
				\SetGraphUnit{5}
				\GraphInit[vstyle=Normal]
				\Vertex[x=0,y=0,Math]{1}
				\Vertex[x=10,y=0,Math]{w_1}
				\Vertex[x=0,y=6,Math]{u_1}
				\Vertex[x=0,y=-6,Math]{u_2}
				\Vertex[x=-10,y=0,Math]{w_2}
				\SetUpEdge[style={->,ultra thick},color=black]
				\Edge[label=$k$](1)(w_1)
				\Edge[label=$k$](1)(w_2)
				\SetUpEdge[style={-> ,ultra thick},color=red]
				\foreach \v in {u_1,u_2}{%
					\Edge[label=$i$](1)(\v)}
				\SetUpEdge[style={-> ,ultra thick},color=blue]
				\foreach \v in {u_1,u_2}{%
					\Edge[label=$j$](\v)(w_1)
					\Edge[label=$j$](\v)(w_2)}
			\end{tikzpicture}
			\caption{A diamond pair in $\mathfrak{X}$.}
		\end{figure}
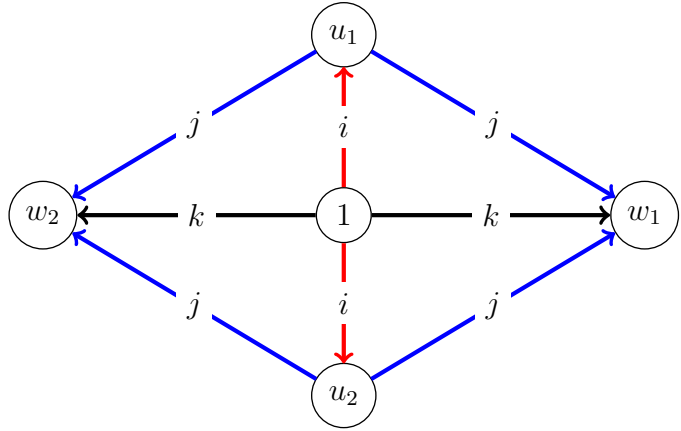
		By definition of the adjacency in the digraphs $(\Gamma_i)$, we have $\Delta_j = \{w_1u_1^{-1},w_1u_2^{-1}\}$. On the one hand, since $\Delta_j$ is an orbit of $G_1 = \langle \varphi\rangle$, we have $\left(w_1u_1^{-1}\right)^\varphi = w_1u_2^{-1}$. On the other hand, using the fact that $G_1$ acts as an automorphism of $R$, we have
		\begin{align*}
			w_1u_2^{-1} = \left(w_1u_1^{-1}\right)^\varphi = w_1^\varphi \left(u_1^{-1}\right)^\varphi = w_2(u_1^\varphi)^{-1} = w_2u_2^{-1}.
		\end{align*}
		Consequently, we have $w_1 = w_2$ which is impossible since $|\Delta_{k}| = 2$. Hence, $\mathfrak{X}$ does not admit any diamond pair, and is therefore triply transitive by Theorem~\ref{thm:main1}. This completes the proof.
	\end{proof}
	
	As an application of Theorem~\ref{thm:main}, we provide the next two examples.
	
	\begin{exa}
		Consider the group $G = \operatorname{D}_{2n} = \langle a,b|a^n=b^2=baba = 1\rangle$, $n\geq 3$, acting on cosets of the subgroup $\langle b\rangle$ by right multiplication. This group admits a regular subgroup, which is the cyclic group of order $n$, and the point stabilizer is a cyclic group of order $2$. By Theorem~\ref{thm:main}, the corresponding association scheme\footnote{this is in fact the association scheme of the distance-regular graph corresponding to the cycle of length $n$.} is triply transitive.\label{exa1}
	\end{exa}
	
	\begin{exa}
		Consider the group $G = \sym{n} = \alt{n} \rtimes \langle (1\ 2)\rangle$, $n\geq 3$, acting on cosets of the subgroup $\langle (1\ 2)\rangle$ by right multiplication. No conjugates of $\langle (1\ 2)\rangle$ can be contained in $\alt{n}$, so the subgroup $\alt{n}$ acts regularly. Hence, again by Theorem~\ref{thm:main}, the association scheme corresponding to this action is triply transitive.\label{exa2}
	\end{exa}
	
	\subsection*{Acknowledgement} We would like to thank the anonymous reviewers for their insightful feedback.
	

\begin{thebibliography}{10}
		
		\bibitem{BannaiMunemasa95}
		E.~Bannai and A.~Munemasa.
		\newblock The {T}erwilliger algebras of group association schemes.
		\newblock {\em Kyushu J. Math.}, 49(1):93--102, 1995.
		
		\bibitem{chen2025structure}
		Z.~Chen and C.~Xi.
		\newblock Structure of {T}erwilliger algebras of quasi-thin association
		schemes.
		\newblock {\em J. Combin. Theory Ser. A}, 213:106024, 2025.
		
		\bibitem{evdokimov2000separability}
		S.~Evdokimov and I.~Ponomarenko.
		\newblock Separability number and {S}churity number of coherent configurations.
		\newblock {\em Electronic J. Combin.}, 7:R31--R31, 2000.
		
		\bibitem{guotriply}
		K.~J. Guo.
		\newblock Triply regular graphs.
		
		\bibitem{hanaki2023terwilliger}
		A.~Hanaki and M.~Yoshikawa.
		\newblock Terwilliger algebras and some related algebras defined by finite
		connected simple graphs.
		\newblock {\em Discrete Math.}, {\bf 346}(9):113509, 2023.
		
		\bibitem{HMR2025}
		A.~Herman, R.~Maleki, and A.~S. Razafimahatratra.
		\newblock On the {T}erwilliger algebra of the group association scheme of the
		symmetric group $\operatorname{Sym}(7)$.
		\newblock {\em J. Combin. Des.}, 2025.
		
		\bibitem{Herman2026}
		A.~Herman, R.~Maleki, and A.~S. Razafimahatratra.
		\newblock On the classification of triply transitive strongly regular graphs.
		\newblock {\em J. Algebraic Combin.}, 63(2):1--49, Mar. 2026.
		
		\bibitem{Hirasaka2002}
		M.~Hirasaka and M.~E. Muzychuk.
		\newblock Association schemes generated by a non-symmetric relation of valency
		2.
		\newblock {\em Discret. Math.}, 244(1-3):109--135, 2002.
		
		\bibitem{Hirasaka2002a}
		M.~Hirasaka and M.~E. Muzychuk.
		\newblock On quasi-thin association schemes.
		\newblock {\em J. Comb. Theory {A}}, 98(1):17--32, 2002.
		
		\bibitem{jiang2023terwilliger}
		Y.~Jiang.
		\newblock On {T}erwilliger ${F}$-algebras of quasi-thin association schemes.
		\newblock {\em J. Algebraic Combin.}, {\bf 57}(4):1219--1251, 2023.
		
		\bibitem{maleki2024terwilliger}
		R.~Maleki.
		\newblock On the {T}erwilliger algebra of the group association scheme of
		${C}_n\rtimes {C}_2$.
		\newblock {\em Discrete Math.}, 347(2):113773, 2024.
		
		\bibitem{HM}
		I.~Miyamoto and A.~Hanaki.
		\newblock List of quasi-thin and non {S}churian schemes.
		\newblock
		\url{http://math.shinshu-u.ac.jp/~hanaki/as/data/QuasiThin_nonSchurian}.
		\newblock Accessed: September 25, 2025.
		
		\bibitem{muzychuk2016terwilliger}
		M.~Muzychuk and B.~Xu.
		\newblock Terwilliger algebras of wreath products of association schemes.
		\newblock {\em Linear Algebra Appl.}, 493:146--163, 2016.
		
		\bibitem{song2017combinatorial}
		S.~Y. Song, B.~Xu, and S.~Zhou.
		\newblock Combinatorial extensions of {T}erwilliger algebras and wreath
		products of association schemes.
		\newblock {\em Discrete Math.}, 340(5):892--905, 2017.
		
		\bibitem{terwilliger1992subconstituent}
		P.~Terwilliger.
		\newblock The subconstituent algebra of an association scheme, (part {I}).
		\newblock {\em J. Algebraic Combin.}, 1(4):363--388, 1992.
		
		\bibitem{terwilliger1993subconstituent2}
		P.~Terwilliger.
		\newblock The subconstituent algebra of an association scheme (part {II}).
		\newblock {\em J. Algebraic Combin.}, 2:73--103, 1993.
		
		\bibitem{terwilliger1993subconstituent3}
		P.~Terwilliger.
		\newblock The subconstituent algebra of an association scheme (part {III}).
		\newblock {\em J. Algebraic Combin.}, 2:177--210, 1993.
		
		\bibitem{sagemath}
		{The Sage Developers}.
		\newblock {\em {S}ageMath, the {S}age {M}athematics {S}oftware {S}ystem
			({V}ersion 10.6)}, 2025.
		\newblock \url{ https://www.sagemath.org}.
		
		\bibitem{yang2025terwilliger}
		J.~Yang, Q.~Guo, X.~Zhang, and L.~Feng.
		\newblock The {T}erwilliger algebras of the group association schemes of two
		non-abelian groups.
		\newblock {\em Comm. Algebra}, {\bf 53}(7):2949--2965, 2025.
		
		\bibitem{yang2024terwilliger}
		J.~Yang, X.~Zhang, and L.~Feng.
		\newblock The {T}erwilliger algebras of the group association schemes of three
		metacyclic groups.
		\newblock {\em J. Combin. Des.}, 32(8):438--463, 2024.
		
	\end{thebibliography}

\end{document}